\documentclass[12pt,a4]{amsart}
\usepackage[a4paper, left=26mm, right=26mm, top=28mm, bottom=34mm]{geometry}

\usepackage{amsmath, amscd}
\usepackage{amssymb}
\usepackage[alphabetic]{amsrefs}
\usepackage[T1]{fontenc}
\usepackage[all]{xy}
\usepackage{mathtools}
\usepackage{footmisc}
\usepackage{stmaryrd}
\usepackage{mathrsfs}

\usepackage{tikz}
\usetikzlibrary{calc}
\usetikzlibrary{matrix,arrows,decorations.pathmorphing}
\usepackage{tikz-cd}

\usepackage{relsize} 
\usepackage[bbgreekl]{mathbbol} 
\usepackage{amsfonts} 
\usepackage[colorlinks=true, hyperindex, linkcolor=cyan, citecolor=, pagebackref=false, urlcolor=cyan, linktocpage]{hyperref}

\DeclareSymbolFontAlphabet{\mathbb}{AMSb} %to ensure that the meaning of \mathbb does not change
\DeclareSymbolFontAlphabet{\mathbbl}{bbold}

\theoremstyle{definition}
\newtheorem{thm}{Theorem}[section]
\newtheorem{mainthm}{Theorem}

\newtheorem{lem}[thm]{Lemma}
\newtheorem{cor}[thm]{Corollary}
\newtheorem{prop}[thm]{Proposition}

\theoremstyle{definition}

\newtheorem{rem}[thm]{Remark}

\newtheorem{dfn}[thm]{Definition}

\newtheorem{ques}[thm]{Question}

\newtheorem{construction}[thm]{Construction}

\newcommand{\bC}{\mathbb{C}}

\newcommand{\bG}{\mathbb{G}}

\newcommand{\bL}{\mathbb{L}}
\newcommand{\bN}{\mathbb{N}}
\newcommand{\bZ}{\mathbb{Z}}

\newcommand{\bQ}{\mathbb{Q}}

\newcommand{\cA}{\mathcal{A}}

\newcommand{\cE}{\mathcal{E}}

\newcommand{\cM}{\mathcal{M}}

\newcommand{\cO}{\mathcal{O}}
\newcommand{\cP}{\mathcal{P}}
\newcommand{\cQ}{\mathcal{Q}}
\newcommand{\sC}{\mathscr{C}}

\newcommand{\isom}{\stackrel{\sim}{\to}}

\newcommand{\et}{\text{\'{e}t}}

\newcommand{\rk}{\mathrm{rk}}

\newcommand{\sHom}{\mathcal{H}om}

\title{Log $p$-divisible groups associated with semi-abelian degeneration}

\author{Kentaro Inoue}
\address{Department of Mathematics, Faculty of Science, Kyoto University, Kyoto 606-8502, Japan}
\email{keninoue0123@gmail.com}

\begin{document}

\maketitle

\begin{abstract}
Kato and Zhao proved that, when an abelian variety $A$ over a complete discrete valuation field $K$ has semi-abelian degeneration $G$ over the valuation ring $\cO_{K}$ of $K$, the associated $p$-divisible group $A[p^{\infty}]$ uniquely extends to a log $p$-divisible group $A[p^{\infty}]^{\mathrm{log}}$ over  $\cO_{K}$. Here, the log $p$-divisible group $A[p^{\infty}]^{\mathrm{log}}$ captures more information than a system of quasi-finite flat group schemes $\{G[p^{n}]\}_{n\geq 1}$ over $\cO_{K}$. In this paper, we generalize the results of Kato and Zhao to the case where the base log scheme is a regular scheme equipped with the log structure defined by a normal crossings divisor.
\end{abstract}
\tableofcontents

\section{Introduction}

It is an important to understand degeneration of abelian varieties. Among such degenerating objects, semi-abelian degeneration of abelian varieties is of particular importance, as is suggested by the semi-stable reduction theorem (\cite{sga7}) and the theory of toroidal compactifications of moduli spaces of abelian varieties (\cite{fc90}). In this paper, we study the behavior of torsion subgroups of abelian varieties degenerating to semi-abelian schemes.

We begin with a simple situation. Let $K$ be a complete discrete valuation field of characteristic $0$ with valuation ring $\cO_{K}$ whose residue field $k$ is of characteristic $p>0$. Let $A$ be a semi-abelian scheme over $\cO_{K}$ with $A_{K}\coloneqq A\otimes_{\cO_{K}} K$ being an abelian variety. The associated $p$-divisible group $A_{K}[p^{\infty}]$ over $K$ degenerates to a system of quasi-finite flat group schemes $A[p^{\infty}]\coloneqq\{A[p^{n}]\coloneqq \mathrm{Ker}(\times p^{n}\colon A\to A)\}_{n\geq 1}$ over $\cO_{K}$. However, this object does not behave well because $A[p^{n}]$ is not necessarily finite flat over $\cO_{K}$ and its rank is not constant unless $A$ is an abelian scheme over $\cO_{K}$. 

As a remedy, Kato used the framework of log geometry developed in \cite{kat89} and introduced the notion of \emph{log finite group schemes} and \emph{log $p$-divisible groups} over fs log schemes (\cite{kat23}). These objects behave like usual finite flat group schemes or $p$-divisible groups, as, for example, suggested by a log version of Dieudonn\'{e} theory (\cites{kat23,ino25}). Furthermore, Kato and Zhao proved that the $p$-divisible group $A_{K}[p^{\infty}]$ over $K$ uniquely extends to a log $p$-divisible group over $\cO_{K}$ equipped with the standard log structure in \cite[Theorem 5.2]{zha21} and \cite[\S 4.3]{kat23}. 

The goal of this paper is to generalize the results of Kato and Zhao to semi-abelian schemes over higher dimensional bases. Our main theorems are as follows.

\begin{mainthm}[Theorem \ref{mainthm BT ver}]\label{mainthmb}
    Let $(X,\cM_{X})$ be an fs log scheme defined by a locally noetherian regular scheme $X$ with a normal crossings divisor $D$, and $A$ be a semi-abelian scheme over $X$. Let $U\coloneqq X-D$. Suppose that $A_{U}$ is an abelian scheme over $U$. Then the $p$-divisible group $A_{U}[p^{\infty}]$ over $U$ uniquely extends to a log $p$-divisible group $A[p^{\infty}]^{\mathrm{log}}$ over $(X,\cM_{X})$. 
\end{mainthm}

We also prove an analogue of Theorem \ref{mainthmb} for $n$-torsion points of $A_{U}$ when $n$ is invertible on the generic points of $D$. This assumption is necessary in order to ensure the uniqueness of the extension of $A_{U}[n]$.

\begin{mainthm}[Theorem \ref{mainthm fin ver}]\label{mainthma}
    Let $n\geq 1$ be an integer. Let $(X,\cM_{X})$ be an fs log scheme defined by a locally noetherian regular scheme $X$ with a normal crossings divisor $D$, and $A$ be a semi-abelian scheme over $X$. Let $U\coloneqq X\backslash D$. Suppose that $A_{U}$ is an abelian scheme over $U$ and that $D\otimes_{\bZ} \bZ[1/n]$ is dense in $D$. Then the finite flat group scheme $A_{U}[n]$ over $U$ uniquely extends to a log finite group scheme $A[n]^{\mathrm{log}}$ over $(X,\cM_{X})$.
\end{mainthm}

Let us explain the relation of our results with the theory of \emph{log abelian varieties} introduced by Kajiwara-Kato-Nakayama in \cite[Definition 4.1]{kkn08b}, which we do not use in this paper. For simplicity, we restrict to Theorem \ref{mainthmb}. Roughly speaking, a log abelian variety is a degenerating object of a usual abelian variety in the world of log geometry. In \cite[Proposition 18.1]{kkn15} and \cite[Proposition 4.5]{kat23}, Kajiwara-Kato-Nakayama proved that, for an integer $n\geq 1$ and a log abelian scheme $A^{\mathrm{log}}$ over an fs log scheme $(Y,\cM_{Y})$, the object $A^{\mathrm{log}}[n]\coloneqq \mathrm{Ker}(\times n\colon A^{\mathrm{log}}\to A^{\mathrm{log}})$ is a log finite group scheme over $(Y,\cM_{Y})$. Hence, Theorem \ref{mainthmb} follows immediately from the results of Kajiwara-kato-Nakayama when $A$ is the semi-abelian part of a log abelian scheme over $(X,\cM_{X})$ in the sense of \cite[4.4]{kkn08b}. It seems to be natural to ask the following question.

\begin{ques}\label{conj on log ab red}
     Let $(X,\cM_{X})$ and $A$ be as in Theorem \ref{mainthmb}. Then there is a unique log abelian scheme $A^{\mathrm{log}}$ over $(X,\cM_{X})$ whose semi-abelian part is isomorphic to $A$.
\end{ques}

\begin{rem}
Question \ref{conj on log ab red} is answered affirmatively in the following cases.
    \begin{enumerate}
    
    \item When $(X,\cM_{X})$ is a spectrum of a complete discrete valuation ring equipped with the standard log structure, Question \ref{conj on log ab red} is resolved affirmatively in \cite[Corollary 4.5]{kkn19}. 
    \item Let $n\geq 3$ be an integer. Let $\cA_{g,n}$ denote the moduli space of principally polarized $g$-dimensional abelian varieties with level-$n$ structures, and $\cA_{g,n}^{\Sigma}$ denote the toroidal compactification of $\cA_{g,n}$ associated with a fixed smooth cone decomposition $\Sigma$ constructed by Faltings-Chai in \cites{fc90}. We equip $\cA_{g,n}^{\Sigma}$ with the log structure $\cM_{\cA_{g,n}^{\Sigma}}$ defined by the boundary divisor. In a series of papers \cites{kkn08a,kkn08b,kkn15,kkn18,kkn19,kkn21,kkn22}, Kajiwara-Kato-Nakayama studied fundamental properties of log abelian varieties and interpreted the log scheme $(\cA_{g,n}^{\Sigma},\cM_{\cA_{g,n}^{\Sigma}})$ as the moduli space of log abelian varieties. In particular, their results imply that there exists a universal log abelian scheme $A^{\mathrm{log}}_{\mathrm{univ}}$ on $(\cA_{g,n}^{\Sigma},\cM_{\cA_{g,n}^{\Sigma}})$ such that its semi-abelian part is isomorphic to the universal semi-abelian scheme $A^{\mathrm{sab}}_{\mathrm{univ}}$. 

    Let $(X,\cM_{X})$ and $A$ be as in Theorem \ref{mainthmb}. Suppose that there is a morphism $X\to \cA_{g,n}^{\Sigma}$ such that $A$ is obtained as the pullback of $A^{\mathrm{sab}}_{\mathrm{univ}}$ along this morphism. Then it follows that Question \ref{conj on log ab red} has an affirmative answer from the functoriality of the semi-abelian parts of log abelian schemes. For example, this assumption holds for the toroidal compactification of the integral canonical model of a Shimura variety of Hodge type with hyperspecial level constructed by Lan and Madapusi in \cites{lan13,mad19}.
    \item When $X$ is a smooth variety over the complex number field $\bC$ and $A$ admits a polarization, Question \ref{conj on log ab red} is resolved affirmatively in \cite[Proposition 3.9.2]{kkn08a} by using log Hodge theory.
    \end{enumerate}
\end{rem}

\begin{rem}
    By the argument in the previous paragraph of Question \ref{conj on log ab red}, some previous works directly imply Theorem \ref{mainthmb} when Question \ref{conj on log ab red} has an affirmative answer. However, let us notice that our methods in this paper do not use the theory of log abelian varieties and is much simpler than an approach using the theory developed by Kajiwara-Kato-Nakayama, even if Question \ref{conj on log ab red} has an affirmative answer.
\end{rem}

In what follows, we explain the strategy of our proof of Theorem \ref{mainthma}. Theorem \ref{mainthmb} is proved in a similar way. We may assume that $X=\mathrm{Spec}(R)$ is affine. First, we consider the case where $R$ is a complete regular local ring. In this case, Mumford's degeneration theory (cf.~\cite{fc90}) allows us to associate a log $1$-motive with a semi-abelian degeneration $A$. Then the log finite group scheme associated with this log $1$-motive constructed in Proposition \ref{log fin ass to log 1mot}) is the desired object. In this step, we prove the following.

\begin{mainthm}[Theorem \ref{log deg theory}]\label{mainthmc}
    Let $X$ be a spectrum of a complete regular local ring and $U$ be an open subset of $X$ whose complement is a normal crossings divisor. Then there are natural equivalences of categories:
    \[
    \mathrm{DEG}(X,U)\simeq \mathrm{DD}(X,U)\simeq \mathrm{DD}^{\mathrm{log}}(X,U).
    \]
    Here, the categories $\mathrm{DEG}(X,U)$, $\mathrm{DD}(X,U)$, and $\mathrm{DD}^{\mathrm{log}}(X,U)$ are defined in Definition \ref{categories concerning degeneration}.
\end{mainthm}

While the first equivalence is already known in \cite{fc90}, the second one is new and allows us to interpret the degeneration theory in terms of the theory of log $1$-motives. We believe that Theorem \ref{mainthmc} is well-known to experts.

Next, we treat the case where $R$ is a (not necessarily complete) discrete valuation ring. In this case, Theorem \ref{mainthma} follows from what we proved above and Beauville-Laszlo gluing (Proposition \ref{bl gluing for log fin grp}). 

Finally, we consider general cases. By the results in the previous cases and the limit argument, we obtain an extension $A_{V}[n]^{\mathrm{log}}$ of $A_{U}[n]$ to an open subset $V\subset X$ containing $U$ such that the codimension of the complement of $V$ in $X$ is at least $2$. The purity for homomorphisms of log finite group schemes (Corollary \ref{purity for hom of weak log fin grp}) and the limit argument allow us to assume that $R$ is local. Let $\widehat{R}$ be the completion of $R$. Consider a strict fpqc cover $(\mathrm{Spec}(\widehat{R}),\cM_{\widehat{R}})\to (X,\cM_{X})$. What we proved above implies that there is a log finite group scheme $A_{\widehat{R}}[n]^{\mathrm{log}}$ over $(\mathrm{Spec}(\widehat{R}),\cM_{\widehat{R}})$ extending the pullback of $A_{V}[n]^{\mathrm{log}}$, and $A_{\widehat{R}}[n]^{\mathrm{log}}$ is naturally equipped with a descent datum over $(\mathrm{Spec}(\widehat{R}\otimes_{R} \widehat{R}),\cM_{\widehat{R}\otimes_{R} \widehat{R}})$ by the purity for homomorphisms of log finite group schemes (Corollary \ref{purity for hom of weak log fin grp}). Therefore, we obtain the desired extension by strict fpqc descent for log finite group schemes (\cite[Theorem 7.1 and Theorem 8.1]{kat21}).

Let us explain the outline of this paper. In Section \ref{sec2}, we recall some basic properties of log $p$-divisible groups. In Section \ref{sec3}, we give the reinterpretation of Mumford's degeneration theory of abelian schemes in terms of log $1$-motives. In Section \ref{sec4}, we prove the main theorem (Theorem \ref{mainthm fin ver}).

\subsection*{Acknowledgements}
The author is grateful to his advisor, Tetsushi Ito, for useful discussions and warm encouragement. The author would like to thank Kazuya Kato and Peihang Wu for some helpful discussions. This work was supported by JSPS KAKENHI Grant Number 23KJ1325.

\subsection*{Notation and conventions}\noindent

\begin{itemize}
\item All rings and monoids are commutative.
\item For a monoid $P$ and an integer $n\geq 1$, let $P^{1/n}$ denote the monoid $P$ with $P\to P^{1/n}$ mapping $p$ to $p^{n}$. The colimit of $P^{1/n}$ with respect to $n\geq 1$ is denoted by $P_{\bQ_{\geq 0}}$.
\item For a log scheme $(S,\cM_{S})$ and a scheme $T$ over $S$, the pullback log structure of $\cM_{S}$ to $T$ is denoted by $\cM_{T}$ unless otherwise specified. 
\item For a site $\sC$, the associated topos with $\sC$ is denoted by $\mathrm{Shv}(\sC)$.
\end{itemize}

We refer readers to \cite{ogu18} for notation and terminologies concerning log schemes.

\section{Preliminaries on log schemes and log $p$-divisible groups}\label{sec2}

\subsection{Kfl vector bundles}

We recall some basic results on kfl topology introduced by Kato in \cite[Definition 2.3]{kat21}.

\begin{dfn}
    A monoid map $f\colon M\to N$ of fs monoids is called \emph{Kummer} if $f$ is injective and, for every $q\in N$, there exist an integer $n\geq 1$ and $p\in M$ such that $f(p)=q^{n}$. 
\end{dfn}

\begin{dfn}[{\cite[(1.10) and Definition 2.2]{kat21}}]
    Let $f\colon (X,\cM_{X})\to (Y,\cM_{Y})$ be a morphism of fs log schemes. 
    \begin{enumerate}
        \item The morphism $f$ is \emph{log flat} (resp. \emph{log \'{e}tale}) if, fppf locally on $X$ and $Y$, there exists a chart $P\to Q$ of $f$ such that the following conditions are satisfied:
        \begin{itemize}
            \item the induced map $P^{\mathrm{gp}}\to Q^{\mathrm{gp}}$ is injective (resp. injective and its cokernel is a finite abelian group with an order invertible on $X$);
            \item the induced morphism $(X,\cM_{X})\to (Y,\cM_{Y})\times_{(\bZ[P],P)^{a}} (\bZ[Q],Q)^{a}$ is strict flat (resp. strict \'{e}tale).
        \end{itemize}
        \item The morphism $f$ is \emph{Kummer} if, for each $x\in X$, the natural map $\cM_{Y,\overline{y}}/\cO_{Y,\overline{y}}^{\times}\to \cM_{X,\overline{x}}/\cO_{X,\overline{x}}^{\times}$ is Kummer, where $y\coloneqq f(x)$.
    \end{enumerate}
\end{dfn}

Let $(X,\cM_{X})_{\mathrm{kfl}}$ (resp. $(X,\cM_{X})_{\mathrm{k\et}}$) be the category of fs log schemes over $(X,\cM_{X})$ equipped with Kummer log flat topology (resp. Kummer log \'{e}tale topology) (\cite[Definition 2.3]{kat21}), called \emph{kfl topology} (resp. \emph{k\'{e}t topology}) for short. The kfl topology is subcanonical (\cite[Theorem 3.1]{kat21}). In other words, for an fs log scheme $(Z,\cM_{Z})$ over $(X,\cM_{X})$, the presheaf on $(X,\cM_{X})_{\mathrm{kfl}}$ given by $(Y,\cM_{Y})\mapsto \mathrm{Mor}_{(X,\cM_{X})}((Y,\cM_{Y}),(Z,\cM_{Z}))$ is a sheaf. In particular, we have a sheaf on $(X,\cM_{X})_{\mathrm{kfl}}$ defined by $(Y,\cM_{Y})\mapsto \Gamma(Y,\cO_{Y})$, denoted by $\cO_{(X,\cM_{X})}$. Furthermore, we define $\bG_{m,\mathrm{log}}$ as the strict \'{e}tale sheafification of the presheaf on $(X,\cM_{X})_{\mathrm{kfl}}$ given by $(Y,\cM_{Y})\mapsto \Gamma(Y,\cM_{Y})^{\mathrm{gp}}$. Then $\bG_{m,\mathrm{log}}$ is a sheaf on $(X,\cM_{X})_{\mathrm{kfl}}$ (\cite[Theorem 3.2]{kat21}).

We refer to vector bundles on the ringed site $((X,\cM_{X})_{\mathrm{kfl}},\cO_{(X,\cM_{X})})$ as \emph{kfl vector bundles} on $(X,\cM_{X})$, and the category of kfl vector bundles on $(X,\cM_{X})$ is denoted by $\mathrm{Vect}(X,\cM_{X})$. For a (usual) vector bundle $\cE$ on $X$, we define the $\cO_{(X,\cM_{X})}$-module $\iota(\cE)$ by $(Y,\cM_{Y})\mapsto \Gamma(Y,f^{*}\cE)$, where $f\colon Y\to X$ is the structure morphism. The sheaf property of $\cO_{(X,\cM_{X})}$ implies that $\iota(\cE)$ is a kfl vector bundle on $(X,\cM_{X})$. Then $\cE\mapsto \iota(\cE)$ gives a fully faithful functor 
\[
\iota\colon \mathrm{Vect}(X)\to \mathrm{Vect}_{\mathrm{kfl}}(X,\cM_{X}),
\]
where $\mathrm{Vect}(X)$ denotes the category of (usual) vector bundles on $X$. We regard $\mathrm{Vect}(X)$ as a full subcategory of $\mathrm{Vect}_{\mathrm{kfl}}(X,M_{X})$ via the functor $\iota$, and, by abuse of notation, we simply write $\cE$ for $\iota(\cE)$. For $\cE\in \mathrm{Vect}_{\mathrm{kfl}}(X,\cM_{X})$, we say that $\cE$ is \emph{classical} if $\cE$ belongs to $\mathrm{Vect}(X)$.

\begin{lem}[{\cite[Lemma 2.4]{ino23}}]\label{kfl vect bdle is classical after n-power ext}
    Let $(X,\cM_{X})$ be a quasi-compact fs log scheme and $\cE$ be a kfl vector bundle on $(X,\cM_{X})$. Suppose that we are given an fs chart $P\to \cM_{X}$. Then the pullback of $\cE$ by a kfl covering
    \[
    (X,\cM_{X})\otimes_{(\bZ[P],P)^{a}} (\bZ[P^{1/n}],P^{1/n})^{a}\to (X,\cM_{X})
    \]
    is classical for some $n\geq 1$. 
\end{lem}

\begin{prop}[Beauville-Laszlo gluing for kfl vector bundles]\label{bl glueing for kfl vect bdle}
Let $(\mathrm{Spec}(R),\cM_{R})$ be a spectrum of a discrete valuation ring $R$ equipped with the log structure defined by the unique closed point. Let $K$ be the fraction field of $R$, $\widehat{R}$ be the completion of $R$, and $\widehat{K}$ be the fraction field of $\widehat{R}$. Then a natural functor
\[
\mathrm{Vect}_{\mathrm{kfl}}(\mathrm{Spec}(R),\cM_{R})\to \mathrm{Vect}(K)\times_{\mathrm{Vect}(\widehat{K})} \mathrm{Vect}_{\mathrm{kfl}}(\mathrm{Spec}(\widehat{R}),\cM_{\widehat{R}})
\]
is an equivalence of categories.    
\end{prop}

\begin{proof}
    Fix a uniformizer $\pi\in R$, and let $\alpha\colon \bN\to \cM_{R}$ be a chart defined by $1\mapsto \pi$. For an integer $n\geq 1$, we set
    \[
    (\mathrm{Spec}(R^{(0)}_{n}),\cM_{R^{(0)}_{n}})\coloneqq (\mathrm{Spec}(R),\cM_{R})\otimes_{(\bZ[\bN],\bN)} (\bZ[\frac{1}{n}\bN],\frac{1}{n}\bN),
    \]
    and we let $(\mathrm{Spec}(R^{(m)}_{n}),\cM_{R^{(m)}_{n}})$ denote the $(m+1)$--fold self-product of $(\mathrm{Spec}(R^{(0)}_{n}),\cM_{R^{(0)}_{n}})$ in the category of saturated log schemes over $(\mathrm{Spec}(R),\cM_{R})$ for $m\geq 0$. Let $K^{(m)}_{n}\coloneqq R^{(m)}_{n}[1/\pi]$, and let $\widehat{R}^{(m)}_{n}$ denote the $\pi$-adic completion of $R^{(m)}_{n}$. Let $\widehat{K}^{(m)}_{n}\coloneqq \widehat{R}^{(m)}_{n}[1/\pi]$. The ring $R^{(0)}_{n}$ is a discrete valuation ring, and $R^{(m)}_{n}$ is flat over $R^{(0)}_{n}$ for $m\geq 1$. Hence, $R^{(m)}_{n}$ is $\pi$-torsion free for $m\geq 0$. Beauville-Laszlo gluing (\cite[Theorem and Remarques (1)]{bl95}) gives equivalences
    \[
    \mathrm{Vect}(R^{(m)}_{n})\isom \mathrm{Vect}(K^{(m)}_{n})\times_{\mathrm{Vect}(\widehat{K}^{(m)}_{n})} \mathrm{Vect}(\widehat{R}^{(m)}_{n})
    \]
    for $m\geq 0$. Therefore, by working kfl locally, we obtain an equivalence
    \[
    \mathrm{Vect}_{\mathrm{kfl},n}(\mathrm{Spec}(R),\cM_{R})\to \mathrm{Vect}(K)\times_{\mathrm{Vect}(\widehat{K})} \mathrm{Vect}_{\mathrm{kfl},n}(\mathrm{Spec}(\widehat{R}),\cM_{\widehat{R}}),
    \]
    where $\mathrm{Vect}_{\mathrm{kfl},n}(\mathrm{Spec}(R),\cM_{R})$ (resp. $\mathrm{Vect}_{\mathrm{kfl},n}(\mathrm{Spec}(\widehat{R}),\cM_{\widehat{R}})$) is the full subcategory of the category of kfl vector bundles on $(\mathrm{Spec}(R),\cM_{R})$ (resp. $(\mathrm{Spec}(\widehat{R}),\cM_{\widehat{R}})$) consisting of objects which become classical after being pulled back to $(\mathrm{Spec}(R^{(0)}_{n}),\cM_{R^{(0)}_{n}})$ (resp. $(\mathrm{Spec}(\widehat{R}^{(0)}_{n}),\cM_{\widehat{R}^{(0)}_{n}})$). Taking the colimit with respect to $n\geq 1$, we obtain the equivalence in the assertion by Lemma \ref{kfl vect bdle is classical after n-power ext}.
\end{proof}

\subsection{Log finite group schemes}

In this subsection, we recall basic results on log finite group schemes and log $p$-divisible groups introduced by Kato in \cite{kat23}.

For a scheme $X$, let $\mathrm{Fin}(X)$ (resp. $\mathrm{BT}(X)$) denote the category of finite and locally free group schemes (resp. $p$-divisible groups) over $X$. When $X=\mathrm{Spec}(R)$ is affine, we write $\mathrm{Fin}(R)=\mathrm{Fin}(X)$ and $\mathrm{BT}(R)=\mathrm{BT}(X)$.

\begin{dfn}[cf.~{\cite[Definition 1.3 and \S 1.6]{kat23}}]
    Let $(X,\cM_{X})$ be an fs log scheme and $G$ be a sheaf of abelian groups on $(X,\cM_{X})_{\mathrm{kfl}}$.
    \begin{enumerate}
        \item We call $G$ a \emph{weak log finite group scheme} if there exists a kfl covering $\{(U_{i},\cM_{U_{i}})\to (X,\cM_{X})\}_{i\in I}$ such that the restriction of $G$ to $(U_{i},\cM_{U_{i}})_{\mathrm{kfl}}$ belongs to $\mathrm{Fin}(U_{i})$ for each $i\in I$. We let $\mathrm{wFin}(X,\cM_{X})$ denote the category of weak log finite group schemes over $(X,\cM_{X})$. The category $\mathrm{Fin}(X)$ is regarded as the full subcategory of $\mathrm{wFin}(X,\cM_{X})$, and we say that an object $G\in \mathrm{wFin}(X,\cM_{X})$ is \emph{classical} if $G$ belongs to $\mathrm{Fin}(X)$.
        \item For a weak log finite group scheme $G$ over $(X,\cM_{X})$, we set
        \[ G^{*}\coloneqq\sHom_{(X,\cM_{X})_{\mathrm{kfl}}}(G,\bG_{m})
        \]
        (which we call the \emph{Cartier dual} of $G$). We say that $G$ is a \emph{log finite group scheme} if both $G$ and $G^{*}$ are representable by finite Kummer log flat log schemes over $(X,\cM_{X})$. We let $\mathrm{Fin}(X,\cM_{X})$ denote the full subcategory of $\mathrm{wFin}(X,\cM_{X})$ consisting of log finite group schemes over $(X,\cM_{X})$. 
    \end{enumerate}
\end{dfn}

\begin{lem}[Coordinate rings, cf.~{\cite[Proposition 2.15]{kat23}}]\label{str sheaf of log fin grp sch}
    There is a natural equivalence between the category $\mathrm{wFin}(X,\cM_{X})$ to the category of Hopf algebra objects of the monoidal tensor category $\mathrm{Vect}(X,\cM_{X})$.
\end{lem}

\begin{proof}
The functor sending $G$ to $f_{*}\cO_{G}$ gives an equivalence between the category of finite and locally free group schemes over $X$ and the category of Hopf algebra objects of the monoidal tensor category $\mathrm{Vect}(X)$, where $f$ is the structure morphism $G\to X$. This equivalence induces the desired equivalence via kfl descent.
\end{proof}

\begin{prop}[Beauville-Laszlo gluing for log finite group schemes]\label{bl gluing for log fin grp}
Under the notation of Proposition \ref{bl glueing for kfl vect bdle}, natural functors
\begin{align*}
    \mathrm{wFin}(\mathrm{Spec}(R),\cM_{R})&\to \mathrm{Fin}(K)\times_{\mathrm{Fin}(\widehat{K})} \mathrm{wFin}(\mathrm{Spec}(\widehat{R}),\cM_{\widehat{R}}), \\
    \mathrm{Fin}(\mathrm{Spec}(R),\cM_{R})&\to \mathrm{Fin}(K)\times_{\mathrm{Fin}(\widehat{K})} \mathrm{Fin}(\mathrm{Spec}(\widehat{R}),\cM_{\widehat{R}})
\end{align*}
are equivalence of categories.
\end{prop}

\begin{proof}
    The equivalence of the former functor follows from Proposition \ref{bl glueing for kfl vect bdle} and Lemma \ref{str sheaf of log fin grp sch}. Then the former equivalence restricts to the latter equivalence thanks to strict fpqc descent for finite Kummer log flat log schemes (\cite[Theorem 7.1 and Theorem 8.1]{kat21}). Although the statement in \emph{loc. cit.} proves strict fppf descent, the same proof as in \emph{loc. cit.} also shows strict fpqc descent.
\end{proof}

Let $\mathrm{Lcf}((X,\cM_{X})_{\mathrm{k\et}})$ denote the category of locally constant sheaves of finite abelian groups on $(X,\cM_{X})_{\mathrm{k\et}}$. Then we have a natural fully faithful functor
\[
\mathrm{Lcf}((X,\cM_{X})_{\mathrm{k\et}})\hookrightarrow \mathrm{wFin}(X,\cM_{X})
\]
by \cite[Theorem 10.2 (2)]{kat21}.

\begin{lem}\label{log fin grp sch and lcf sheaf}
     Let $n\geq 1$ be an integer that is invertible on $X$. Then the above functor induces equivalences
     \[
     \mathrm{Lcf}((X,\cM_{X})_{\mathrm{k\et}},\bZ/n)\isom \mathrm{Fin}((X,\cM_{X}),\bZ/n)\isom \mathrm{wFin}((X,\cM_{X}),\bZ/n),
     \]
     where $\mathrm{Lcf}((X,\cM_{X})_{\mathrm{k\et}},\bZ/n)$,  $\mathrm{Fin}((X,\cM_{X}),\bZ/n)$, $\mathrm{wFin}((X,\cM_{X}),\bZ/n)$ are the full subcategories of $\mathrm{Lcf}((X,\cM_{X})_{\mathrm{k\et}})$, $\mathrm{Fin}(X,\cM_{X})$, $\mathrm{wFin}(X,\cM_{X})$)) consisting of objects killed by $n$, respectively.
\end{lem}

\begin{proof}
    The equivalence of the second functor follows from \cite[Proposition 2.1]{kat23}. 
    
    Since a finite and locally free group scheme killed by an integer invertible on the base is an \'{e}tale locally constant sheaf, $\mathrm{wFin}((X,\cM_{X}),\bZ/n)$ is nothing but the category of locally constant sheaves of finite $\bZ/n$-modules on $(X,\cM_{X})_{\mathrm{kfl}}$. Hence, it follows from \cite[Theorem 10.2 (2)]{kat21} that the composite of the functors in the statement is an equivalence.
\end{proof}

The notion of log finite group schemes allows us to define log $p$-divisible groups in a usual way.

\begin{dfn}
Let $p$ be a prime, and let $(X,\cM_{X})$ be an fs log scheme. Let $G$ be a sheaf of abelian groups on $(X,\cM_{X})_{\mathrm{kfl}}$. We call $G$ a \emph{weak log $p$-divisible group} if the following conditions are satisfied:
\begin{enumerate}
    \item the multiplication by $p$ map $\times p\colon G\to G$ is surjective;
    \item for every $n\geq 1$, the sheaf $G[p^n]\coloneqq\mathrm{Ker}(\times p^{n}\colon G\to G)$ is a weak log finite group scheme over $(X,\cM_{X})$;
    \item $G=\bigcup_{n\geq 1} G[p^{n}].$
\end{enumerate}
The category of weak log $p$-divisible groups over $(X,\cM_{X})$ is denoted by $\mathrm{wBT}(X,\cM_{X})$. A weak log $p$-divisible group $G$ over $(X,\cM_{X})$ is called a \emph{log $p$-divisible group} if $G[p^n]$ is a log finite group scheme for every $n\geq 1$. The category of log $p$-divisible groups over $(X,\cM_{X})$ is denoted by $\mathrm{BT}(X,\cM_{X})$. The category $\mathrm{BT}(X)$ of $p$-divisible groups over $X$ is regarded as the full subcategory of $\mathrm{wBT}(X,\cM_{X})$. A weak log $p$-divisible group $G$ over $(X,\cM_{X})$ is called \emph{classical} if $G$ belongs to $\mathrm{BT}(X)$. Clearly, $G$ is classical if and only if $G[p^n]$ is classical for every $n\geq 1$. 
\end{dfn}

\subsection{Log regular schemes}

In this subsection, we recall the definition of log regularity and some properties of log regular log schemes.

\begin{dfn}[\cite{kat94}]
    Let $(X,\cM_{X})$ be a locally noetherian fs log scheme. For $x\in X$, let $\bar{x}$ denote a geometric point on $x$. Let $I(\bar{x})$ be the ideal of $\cO_{X, \bar{x}}$ generated by the image of the map $\cM_{X, \bar{x}}\setminus \cO_{X, \bar{x}}^{\times}\to \cO_{X, \bar{x}}$. We say that $(X,\cM_{X})$ is \emph{log regular} at $x$ if the following conditions are satisfied:
    \begin{enumerate}
        \item $\cO_{X, \bar{x}}/I(\bar{x})$ is a regular local ring.
        \item $\dim(\cO_{X, \bar{x}})=\dim(\cO_{X, \bar{x}}/I(\bar{x}))+\rk(\cM_{X, \bar{x}}^{\mathrm{gp}}/\cO_{X, \bar{x}}^{\times})$.
    \end{enumerate}
    The log scheme $(X,\cM_{X})$ is called \emph{log regular} if it is log regular at every point $x\in X$. For example, an fs log scheme $(X,\cM_{X})$ defined by a locally noetherian regular scheme $X$ and a normal crossings divisor $D$ is log regular. Conversely, for a log regular log scheme $(X,\cM_{X})$ whose underlying scheme $X$ is regular, the log structure $\cM_{X}$ is defined by a normal crossings divisor by \cite[Theorem 11.6]{kat94} and \cite[Chapter $\rm{III}$, Theorem 1.11.6]{ogu18}.
    
    For a log regular log scheme $(X,\cM_{X})$, the condition $(2)$ implies that the largest open subset $U$ on which the log structure $\cM_{X}$ is trivial is dense. Such an open subset $U$ is called the \emph{interior} of $(X,\cM_{X})$.
\end{dfn}

\begin{prop}[Kato]\label{log reg}
    Let $(X,\cM_{X})$ be a locally noetherian fs log scheme. 
    \begin{enumerate}
        \item The subset $\{ x\in X \mid (X,\cM_{X}) \ \text{is log regular at} \ x\}\subset X$  is stable under generalization.
        \item If $(X,\cM_{X})$ is log regular at $x\in X$, the scheme $X$ is normal at $x$.
        \item Suppose that $(X,\cM_{X})$ is log regular. Let $U$ be the interior of $(X,\cM_{X})$. Then $\cM_{X}$ is the subsheaf of $\cO_{X}$ consisting of functions invertible on $U$.
    \end{enumerate}
\end{prop}
\begin{proof} (1) See \cite[Proposition 7.1]{kat94}.
    
    (2) See \cite[Theorem 4.1]{kat94}.

    (3) See \cite[Theorem 11.6]{kat94}.
\end{proof}

\begin{lem}\label{inv func on interior}
    Let $(X,\cM_{X})$ be a log regular log scheme with an interior $U\subset X$. Suppose that we are given a finitely generated monoid $P$ and a chart $\alpha\colon P\to \cM_{X}$. Then a natural monoid map $P^{\mathrm{gp}}\oplus \cM_{X}(X)\to \cO_{X}(U)^{\times}$ is surjective.
\end{lem}

\begin{proof}
    Take a generator $\{p_{1},\dots,p_{m}\}$ of $P$, and let $p\coloneqq \prod_{i=1}^{m}p_{i}$. Then the vanishing locus of $\alpha(p)\in \cM_{X}(X)\subset \cO_{X}(X)$ coincides with $X\backslash U$. Let $f\in \cO_{U}(U)^{\times}$. For the generic point $\eta$ of each irreducible component $E$ of $X\backslash U$ with $\mathrm{codim}_{X}(E)=1$, the local ring $\cO_{X,\eta}$ is a discrete valuation ring by Proposition \ref{log reg}(2). Take a sufficiently large integer $N\geq 1$ such that  the valuation of $\alpha(p)^{N}f$ defined by the discrete valuation ring $\cO_{X,\eta}$ is non-negative for every $\eta$. Then $\alpha(p)^{N}f\in \cO_{X}(X)$, and so $\alpha(p)^{N}f\in \cM_{X}(X)$ by Proposition \ref{log reg}(3). This proves the assertion.
\end{proof}

\begin{lem}[{\cite[Lemma 4.3]{ino23}}] \label{log reg stab}
    Let $(X,\cM_{X})$ be a log regular log scheme whose underlying scheme is the spectrum of a noetherian strict local ring. Let $x$ be the unique closed point of $X$. Fix a chart $P\to \cM_{X}$ inducing $P\isom \cM_{X,\bar{x}}/\cO_{X,\bar{x}}^{\times}$. Then, for an fs monoid $Q$ and a Kummer map $P\to Q$, the fs log scheme $(X,\cM_{X})\otimes_{(\bZ[P],P)} (\bZ[Q],Q)$ is also log regular.
\end{lem}

\begin{lem}[{\cite[Lemma 4.4]{ino23}}] \label{log reg dvr}
    Let $(X,\cM_{X})$ be a log regular log scheme whose underlying scheme is the spectrum of a strict local discrete valuation ring. Then the log structure $\cM_{X}$ is either the trivial one or the one defined by the unique closed point of $X$.
\end{lem}

\section{Degeneration theory of abelian schemes}\label{sec3}

The goal of this section is to reinterpret the degeneration theory of abelian schemes established by Mumford, Falting-Chai, and Lan in \cites{fc90,lan13} in terms of log $1$-motives (Proposition \ref{log deg theory}).

\subsection{$1$-motives}

Let $S$ be a base scheme. A commutative group scheme $G$ of finite presentation over $S$ is called a \emph{semi-abelian scheme} if every geometric fiber of $G$ is written as an extension of an abelian variety by a torus. We say that a semi-abelian scheme $G$ is \emph{split} if there is an exact sequence $0\to T\to G\to A\to 0$, where $T$ is a torus and $A$ is an abelian scheme over $S$. This exact sequence is unique up to a unique isomorphism if it exists, and $T$ (resp. $A$) is called the \emph{torus part} (resp. \emph{abelian part}) of $G$.

\begin{dfn}[$1$-motives, {\cite[D\'{e}finition 10.1.2 and Variante 10.1.10]{del74}}]
    A \emph{$1$-motive} over $S$ is a morphism $\cQ=(Y\stackrel{u}{\to} G)$ of \'{e}tale sheaves, where $Y$ is a locally constant sheaf of free $\bZ$-modules of finite rank on $S_{\et}$ and $G$ is a split semi-abelian scheme over $S$. 
\end{dfn}

Let $\cQ=(Y\stackrel{u}{\to} G)$ be a $1$-motive over $S$ and $T$ (resp.~$A$) be the torus part (resp.~the abelian part) of $G$. We let $c\colon Y\to A$ denote the composite of $Y\stackrel{u}{\to} G\twoheadrightarrow A$. Let $X$ denote the character group sheaf of $T$. The extension class corresponding to $G$ belongs to
\[
\mathrm{Ext}^{1}_{S_{\et}}(A,T)\cong \mathrm{Hom}_{S_{\et}}(X,\mathcal{E}xt^{1}_{S_{\et}}(A,\bG_{m}))\cong \mathrm{Hom}_{S_{\et}}(X,A^{\vee}).
\]
This gives a group homomorphism $c^{\vee}\colon X\to A^{\vee}$. Take $x\in X(S)$. Taking the pushout along $x\colon T\to \bG_{m}$ and the pullback along $c\colon Y\to A$ for the exact sequence $0\to T\to G\to A\to 0$ gives an exact sequence
\[
0\to \bG_{m}\to (c\times c^{\vee}(x))^{*}\cP_{A}\to Y\to 0,
\]
where $\cP_{A}$ is the Poincar\'{e} biextension over $A\times A^{\vee}$ and $c\times c^{\vee}(x)$ denotes a map $Y\to A\times A^{\vee}$ given by $y\mapsto c(y)\times c^{\vee}(x)$. Then $u$ induces a section of this exact sequence. By varying $x\in X(S)$, sections defined in this way are totalized into a trivialization of a $\bG_{m}$-biextension over $Y\times X$
\[
\tau\colon 1_{Y\times X}\isom (c\times c^{\vee})^{*}\cP_{A}.
\]
By construction, we can recover the $1$-motive $\cQ$ from the tuple $(X,Y,A,c,c^{\vee},\tau)$. As a summary, we get the following lemma.

\begin{lem}[The description of $1$-motives of a symmetric form]\label{symm description of 1mot}
    Consider the category of tuples $(X,Y,A,c,c^{\vee},\tau)$ consisting of the following objects:
    \begin{itemize}
        \item $X$ and $Y$ are locally constant sheaves of free $\bZ$-modules of finite rank on $S_{\et}$;
        \item $A$ is an abelian scheme over $S$
        \item $c\colon Y\to A$ and $c^{\vee}\colon X\to A^{\vee}$ are group homomorphisms;
        \item $\tau\colon 1_{Y\times X}\isom (c\times c^{\vee})^{*}\cP_{A}$ is a trivialization of a $\bG_{m}$-biextension over $Y\times X$.
    \end{itemize}
    Morphisms $f\colon (X_{1},Y_{1},A_{1},c_{1},c_{1}^{\vee},\tau_{1})\to (X_{2},Y_{2},A_{2},c_{2},c_{2}^{\vee},\tau_{2})$ are group homomorphisms $f^{\mathrm{mult}}\colon X_{2}\to X_{1}$, $f^{\mathrm{ab}}\colon A_{1}\to A_{2}$, and $f^{\et}\colon Y_{1}\to Y_{2}$ satisfying the following conditions:
    \begin{itemize}
        \item $f^{\mathrm{ab}}c_{1}=c_{2}f^{\et}$ and $c_{1}^{\vee}f^{\mathrm{mult}}=(f^{\mathrm{ab}})^{\vee}c_{2}^{\vee}$;
        \item $(\mathrm{id}_{X_{1}}\times f^{\mathrm{mult}})^{*}\tau_{1}$ and $(f^{\et}\times \mathrm{id}_{X_{2}})^{*}\tau_{2}$ are equal via the isomorphisms
        \begin{align*}
            (\mathrm{id}_{X_{1}}\times f^{\mathrm{mult}})^{*}(c_{1}\times c_{1}^{\vee})^{*}\cP_{A_{1}}&\cong (c_{1}\times c_{2}^{\vee})^{*}(\mathrm{id}_{A_{1}}\times (f^{\mathrm{ab}})^{\vee})^{*}\cP_{A_{1}} \\
            &\cong (c_{1}\times c_{2}^{\vee})^{*}(f^{\mathrm{ab}}\times \mathrm{id}_{A_{2}^{\vee}})^{*}\cP_{A_{2}} \\
            &\cong (f^{\et}\times \mathrm{id}_{X_{2}})^{*}(c_{2}\times c_{2}^{\vee})^{*}\cP_{A_{2}}.
        \end{align*}
    \end{itemize}
    Then this category is naturally equivalent to the category of $1$-motives over $S$.
\end{lem}

\begin{rem}
    We also refer to an object of the category in Lemma \ref{symm description of 1mot} as a $1$-motive over $S$.
\end{rem}

\begin{dfn}[Polarization on $1$-motives]\label{def of pol on 1mot}
    Let $\cQ=(X,Y,A,c,c^{\vee},\tau)$ be a $1$-motive over $S$. The tuple $\cQ^{\vee}\coloneqq (Y,X,A^{\vee},c^{\vee},c,\tau^{\vee})$ is called the \emph{dual $1$-motive} of $\cQ$. Here, $\tau^{\vee}$ is the composite of the isomorphisms of $\bG_{m}$-biextensions over $X\times Y$ defined in the following way:
    \[
    1_{X\times Y}\stackrel{s^{*}\tau}{\isom} s^{*}(c\times c^{\vee})^{*}\cP_{A}\cong (c^{\vee}\times c)^{*}t^{*}\cP_{A}\cong (c^{\vee}\times c)^{*}\cP_{A^{\vee}},
    \]
    where $s\colon X\times Y\to Y\times X$ and $t\colon A^{\vee}\times A\to A\times A^{\vee}$ are the switching maps.
    
    Let $T^{\vee}$ denote the torus over $S$ whose character group is $Y$. The group homomorphism $c$ corresponds to a split semi-abelian scheme $G^{\vee}$ with an exact sequence $0\to T^{\vee}\to G^{\vee}\to A^{\vee}\to 0$. Then $\tau^{\vee}$ gives a group homomorphism $u^{\vee}\colon X\to G^{\vee}$, which corresponds to the dual $1$-motive $\cQ^{\vee}$ via the equivalence in Lemma \ref{symm description of 1mot}.
    
    A \emph{polarization} on $\cQ$ is a morphism $\lambda\colon \cQ\to \cQ^{\vee}$ such that the following conditions are satisfied:
    \begin{itemize}
        \item $\lambda^{\mathrm{mult}}=\lambda^{\et}\colon Y\to X$ and these maps induce an isomorphism $Y\otimes_{\bZ} \bQ\isom X\otimes_{\bZ} \bQ$;
        \item $\lambda^{\mathrm{ab}}\colon A\to A^{\vee}$ is a polarization on the abelian scheme $A$.
    \end{itemize}
\end{dfn}

\subsection{Log $1$-motives}

Next, we consider the log version of $1$-motives. Let $(S,\cM_{S})$ be a locally noetherian fs log scheme.

Let $T$ be a torus over $S$ with character group $X$. We define a sheaf $T_{\mathrm{log}}$ on $(S,\cM_{S})_{\mathrm{kfl}}$ by
\[
T_{\mathrm{log}}\coloneqq \mathcal{H}om_{(S,\cM_{S})_{\mathrm{kfl}}}(X,\bG_{m,\mathrm{log}}).
\]
The natural injection $\bG_{m}\hookrightarrow \bG_{m,\mathrm{log}}$ induces an injection $T\hookrightarrow T_{\mathrm{log}}$.
More generally, for a split semi-abelian scheme $G$ over $S$ with torus part $T$ and abelian part $A$, we define the sheaf $G_{\mathrm{log}}$ on $(S,\cM_{S})_{\mathrm{kfl}}$ by the following pushout diagram:
\[
\begin{tikzcd}
    T \ar[r,hook] \ar[d,hook] & T_{\mathrm{log}} \ar[d] \\
    G \ar[r] & G_{\mathrm{log}}.
\end{tikzcd}
\]
Then we have an exact sequence $0\to T_{\mathrm{log}}\to G_{\mathrm{log}}\to A\to 0$ of sheaves on $(S,\cM_{S})_{\mathrm{kfl}}$. 

\begin{lem}
    The restriction to the small \'{e}tale site $S_{\et}$ gives an exact sequence of sheaves on $S_{\et}$
    \[
    0\to (T_{\mathrm{log}})|_{S_{\et}}\to (G_{\mathrm{log}})|_{S_{\et}}\to A\to 0.
    \]
\end{lem}

\begin{proof}
    By working \'{e}tale locally on $S$, we may assume that $T$ is a split torus. Consider the morphism of sites $\epsilon\colon (S,\cM_{S})_{\mathrm{kfl}}\to S_{\et}$ induced from the inclusion functor $S_{\et}\hookrightarrow (S,\cM_{S})_{\mathrm{kfl}}$. By \cite[Theorem 5.1]{kat21}, we have $R^{1}\epsilon_{*}T_{\mathrm{log}}=0$. Therefore, applying $\epsilon_{*}$ to the exact sequence $0\to T_{\mathrm{log}}\to G_{\mathrm{log}}\to A\to 0$ gives the exact sequence in the assertion.
\end{proof}

\begin{dfn}[Log $1$-motives, {\cite[Definition 2.2]{kkn08b}}]
    A \emph{log $1$-motive} over $(S,\cM_{S})$ is a morphism $\cQ_{\mathrm{log}}\coloneqq (Y\stackrel{u}{\to} G_{\mathrm{log}})$ of sheaves of abelian groups on $(S,\cM_{S})_{\mathrm{kfl}}$, where $Y$ is a strict \'{e}tale locally constant sheaf of free $\bZ$-modules of finite rank on $(S,\cM_{S})_{\mathrm{kfl}}$ and $G$ is a split semi-abelian scheme over $S$.
\end{dfn}

For an abelian scheme $A$ over $S$, we let $\cP_{A}^{\mathrm{log}}$ denote the $\bG_{m,\mathrm{log}}$-biextension over $A\times A^{\vee}$ defined as the base change of the Poincar\'{e} biextension $\cP_{A}$ along $\bG_{m}\to \bG_{m,\mathrm{log}}$. In the same way as Lemma \ref{symm description of 1mot}, we get the following lemma.

\begin{lem}[The description of log $1$-motives of a symmetric form]\label{symm desc of log 1mot}
    Consider the category of tuples $(X,Y,A,c,c^{\vee},\tau)$ consisting of the following objects:
    \begin{itemize}
        \item $(X,Y,A,c,c^{\vee})$ is the same as in Lemma \ref{symm description of 1mot};
        \item $\tau\colon 1_{Y\times X}\isom (c\times c^{\vee})^{*}\cP_{A}^{\mathrm{log}}$ is a trivialization of a $\bG_{m,\mathrm{log}}$-biextension over $Y\times X$.
    \end{itemize}
    Morphisms are also defined in the same way as in Lemma \ref{symm description of 1mot}. Then this category is naturally equivalent to the category of log $1$-motives over $(S,\cM_{S})$.
\end{lem}

\begin{rem}
    We also refer to an object of the category in Lemma \ref{symm desc of log 1mot} as a log $1$-motive over $(S,\cM_{S})$.
\end{rem}

\begin{dfn}[Monodromy pairings associated with log $1$-motives, {\cite[(2.3)]{kkn08b}}]\label{monodromy pairing ass with log 1mot}
    Let $\cQ_{\mathrm{log}}=(Y\stackrel{u}{\to} G_{\mathrm{log}})=(X,Y,A,c,c^{\vee},\tau)$ be a log $1$-motive over $(S,\cM_{S})$. We have the following commutative diagram whose rows are exact sequences of sheaves on $(S,\cM_{S})_{\mathrm{kfl}}$:
    \[
    \begin{tikzcd}
        0 \ar[r] & T \ar[r] \ar[d] & G \ar[r] \ar[d] & A \ar[r] \ar[d,equal] & 0 \\
        0 \ar[r] & T_{\mathrm{log}} \ar[r] & G_{\mathrm{log}} \ar[r] & A \ar[r] & 0.
    \end{tikzcd}
    \]
    By the snake lemma, we get an isomorphism $T_{\mathrm{log}}/T\isom G_{\mathrm{log}}/G$. The homomorphism $u$ induces a homomorphism 
    \[
    Y\to G_{\mathrm{log}}\twoheadrightarrow G_{\mathrm{log}}/G\cong T_{\mathrm{log}}/T\cong \mathcal{H}om_{(S,\cM_{S})_{\mathrm{kfl}}}(X,\bG_{m,\mathrm{log}}/\bG_{m}),
    \]
    which corresponds to a bilinear pairing 
    \[
    \langle -,-\rangle\colon Y\times X\to \bG_{m,\mathrm{log}}/\bG_{m}.
    \]
    This pairing is called the \emph{monodromy pairing} associated with $\cQ_{\mathrm{log}}$.
    \end{dfn}

\begin{dfn}[Dual on log $1$-motives, {\cite[Definition 2.7.4]{kkn08b}}]\label{def of pol on log 1mot} \noindent

    Let $\cQ_{\mathrm{log}}=(X,Y,A,c,c^{\vee},\tau)$ be a log $1$-motive over $(S,\cM_{S})$. The tuple $\cQ_{\mathrm{log}}^{\vee}\coloneqq (Y,X,A^{\vee},c^{\vee},c,\tau^{\vee})$ is called the \emph{dual log $1$-motive} of $\cQ_{\mathrm{log}}$. Here, the trivialization $\tau^{\vee}$ of the $\bG_{m}^{\mathrm{log}}$-biextension $(c^{\vee}\times c)^{*}\cP_{A^{\vee}}^{\mathrm{log}}$ over $X\times Y$ is defined in the same way as in Definition \ref{def of pol on 1mot}. Then $\tau^{\vee}$ gives a group homomorphism $u^{\vee}\colon X\to G^{\vee}_{\mathrm{log}}$, which corresponds to the dual log $1$-motive $\cQ^{\vee}_{\mathrm{log}}$ via the equivalence in Lemma \ref{symm desc of log 1mot}.
\end{dfn}

We can associate a log finite group scheme with a log $1$-motive by taking $n$-torsion points in an appropriate sense.

\begin{dfn}[{\cite[Definition 3.4]{wz24}}]
    Let $\cQ_{\mathrm{log}}=(Y\stackrel{u}{\to} G_{\mathrm{log}})$ be a log $1$-motive over $(S,\cM_{S})$. For an integer $n\geq 1$,  consider the sheaf of commutative groups on $(S,\cM_{S})_{\mathrm{kfl}}$ defined by
    \[
    \cQ_{\mathrm{log}}[n]\coloneqq H^{-1}((Y\stackrel{u}{\to} G_{\mathrm{log}})\otimes^{\bL}_{\bZ} \bZ/n),
    \]
    where $(Y\stackrel{u}{\to} G_{\mathrm{log}})$ is regarded as a complex of sheaves on $(S,\cM_{S})_{\mathrm{kfl}}$ such that $Y$ lives in the degree $-1$ part. Concretely, we can write
    \[
    \cQ_{\mathrm{log}}[n]=\frac{\mathrm{Ker}(u-(\times n)\colon Y\oplus G_{\mathrm{log}}\to G_{\mathrm{log}})}{\mathrm{Im}((\times n)+u\colon Y\to Y\oplus G_{\mathrm{log}})}.
    \]
\end{dfn}

\begin{lem}\label{lem for weil pairing}
    Let $\cQ_{\mathrm{log}}=(Y\stackrel{u}{\to} G_{\mathrm{log}})=(X,Y,A,c,c^{\vee},\tau)$ be a log $1$-motive over $(S,\cM_{S})$. Let $\pi\colon G_{\mathrm{log}}\to A$ and $\pi^{\vee}\colon G^{\vee}_{\mathrm{log}}\to A^{\vee}$ be natural surjections. Then there are natural trivializations of $\bG_{m,\mathrm{log}}$-biextensions
    \[
    \rho_{1}\colon 1_{G_{\mathrm{log}}\times X}\isom (\pi\times c^{\vee})^{*}\cP_{A}^{\mathrm{log}}, \ \ \ \rho_{2}\colon 1_{Y\times G^{\vee}_{\mathrm{log}}}\isom (c\times \pi^{\vee})^{*}\cP_{A}^{\mathrm{log}},
    \]
    and we have an equality $(u\times \mathrm{id}_{X})^{*}(\rho_{1})=(\mathrm{id}_{Y}\times u^{\vee})^{*}(\rho_{2})=\tau$.
\end{lem}

\begin{proof}
    Take a section $x\in X(S,\cM_{S})$. We have the following commutative diagram of exact sequences:
    \[
    \begin{tikzcd}
         0 \ar[r] & T_{\mathrm{log}} \ar[r] \ar[d] & G_{\mathrm{log}} \ar[r] \ar[d] & A \ar[r] \ar[d,equal] & 0 \\
         0 \ar[r] & \bG_{m,\mathrm{log}} \ar[r] & \cP_{A}^{\mathrm{log}}|_{A\times \{c^{\vee}(x)\}} \ar[r] & A \ar[r] & 0,
    \end{tikzcd}
    \]
    where the left vertical map is induced from $x\colon T\to \bG_{m}$. Then the middle vertical map gives a trivialization of the $\bG_{m,\mathrm{log}}$-torsor $((\pi\times c^{\vee})^{*}\cP_{A}^{\mathrm{log}})|_{G_{\mathrm{log}}\times \{x\}}$ on $G_{\mathrm{log}}\times \{x\}\cong G_{\mathrm{log}}$. By varying $x$, the trivializations obtained in this way are totalized into a trivialization of $\bG_{m,\mathrm{log}}$-biextensions $\rho_{1}\colon 1_{G_{\mathrm{log}}\times X}\isom (\pi\times c^{\vee})^{*}\cP_{A}^{\mathrm{log}}$. By construction, $(u\times \mathrm{id}_{X})^{*}(\rho_{1})$ coincides with $\tau$. The remaining assertions are also proved in the same way.
\end{proof}

\begin{construction}[Weil pairings associated with log $1$-motives]
    Let $\cQ_{\mathrm{log}}=(Y\stackrel{u}{\to} G_{\mathrm{log}})=(X,Y,A,c,c^{\vee},\tau)$ be a log $1$-motive over $(S,\cM_{S})$. We shall construct a pairing 
    \[
    e_{\cQ_{\mathrm{log}}[n]}\colon \cQ_{\mathrm{log}}[n]\times \cQ_{\mathrm{log}}^{\vee}[n]\to \bG_{m,\mathrm{log}}
    \]
    as follows:
    Let $q_{1}\coloneqq (y,g)\in (Y\times G_{\mathrm{log}})(S,\cM_{S})$ and $q_{2}\coloneqq (x,h)\in (X\times G_{\mathrm{log}}^{\vee})(S,\cM_{S})$ such that $u(y)=ng$ and $u^{\vee}(x)=nh$. Then there is a unique $e_{\cQ_{\mathrm{log}}[n]}(q_{1},q_{2})\in \bG_{m,\mathrm{log}}(S,\cM_{S})$ fitting into the following commutative diagram of $\bG_{m,\mathrm{log}}$-torsors on $(S,\cM_{S})$:
    \[
    \begin{tikzcd}
        \cP_{A}^{\mathrm{log}}|_{(ng,h)} \ar[r,"\sim"] & \cP_{A}^{\mathrm{log}}|_{(u(y),h)} \ar[r,"\rho_{2}(y{,}h)","\sim"'] & \bG_{m,\mathrm{log}} \ar[dd,"e_{\cQ_{\mathrm{log}}[n]}(q_{1}{,}q_{2})"] \\
        (\cP_{A}^{\mathrm{log}}|_{(g,h)})^{\otimes n} \ar[u,"\sim", sloped] \ar[d,"\sim"',sloped] & & \\
        \cP_{A}^{\mathrm{log}}|_{(g,nh)} \ar[r,"\sim"] & \cP_{A}^{\mathrm{log}}|_{(g,u^{\vee}(x))} \ar[r,"\rho_{1}(g{,}x)","\sim"'] & \bG_{m,\mathrm{log}}.
    \end{tikzcd}
    \]
    Here, the left vertical maps are defined by the $\bG_{m,\mathrm{log}}$-biextension structure on $\cP_{A}^{\mathrm{log}}$. The last assertion of Lemma \ref{lem for weil pairing} implies that $(q_{1},q_{2})\mapsto e_{\cQ_{\mathrm{log}}[n]}(q_{1},q_{2})$ induces a pairing $e_{\cQ_{\mathrm{log}}[n]}\colon \cQ_{\mathrm{log}}[n]\times \cQ_{\mathrm{log}}^{\vee}[n]\to \bG_{m}$ (the bilinearity implies the image is contained in $\bG_{m,\mathrm{log}}[n]=\bG_{m}[n]\subset \bG_{m}$). The pairing $e_{\cQ_{\mathrm{log}}[n]}$ is called the \emph{Weil pairing} associated with the log $1$-motive $\cQ_{\mathrm{log}}$.
\end{construction}

\begin{prop}[cf.~{\cite[Proposition 3.5]{wz24}}]\label{log fin ass to log 1mot}
    For a log $1$-motive $\cQ_{\mathrm{log}}=(Y\stackrel{u}{\to} G_{\mathrm{log}})$ over $(S,\cM_{S})$, the following statements hold.
    \begin{enumerate}
        \item $\cQ_{\mathrm{log}}[n]$ fits into an exact sequence
        \[
        0\to G[n]\to \cQ_{\mathrm{log}}[n]\to Y/nY\to 0
        \]
        of sheaves of abelian groups on $(S,\cM_{S})_{\mathrm{kfl}}$.
        \item $\cQ_{\mathrm{log}}[n]$ is a log finite group scheme over $(S,\cM_{S})$.
        \item The Weil pairing $e_{\cQ_{\mathrm{log}}[n]}$ induces an isomorphism $\cQ_{\mathrm{log}}^{\vee}[n]\isom (\cQ_{\mathrm{log}}[n])^{\vee}$ of log finite group schemes.
        \item For another integer $m\geq 1$, there is a natural exact sequence
        \[
        0\to \cQ_{\mathrm{log}}[m]\to \cQ_{\mathrm{log}}[mn]\to \cQ_{\mathrm{log}}[n]\to 0.
        \]
        In particular, for a prime number $p$, $\cQ_{\mathrm{log}}[p^{\infty}]\coloneqq \bigcup_{n\geq 1}\cQ_{\mathrm{log}}[p^{n}]$ is a log $p$-divisible group over $(S,\cM_{S})$.
    \end{enumerate}
\end{prop}

\begin{proof}
(1) is proved by W\"{u}rthen-Zhao in \cite[Proposition 3.5]{wz24}. (2) follows from \cite[Proposition 2.3]{kat23} and (1). (4) follows from (1) and the snake lemma. (3) is presumably well-known to experts. We shall give a proof here because we could not find proofs in the literature.

We define a filtration $W_{-2,\cQ_{\mathrm{log}}}\subset W_{-1,\cQ_{\mathrm{log}}}\subset W_{0,\cQ_{\mathrm{log}}}=\cQ_{\mathrm{log}}[n]$ by
\[
W_{-2,\cQ_{\mathrm{log}}}\coloneqq T[n]\cong X^{\vee}\otimes_{\bZ} \mu_{n}, \ \ \ W_{-1,\cQ_{\mathrm{log}}}\coloneqq G[n].
\]
Applying (1) to $\cQ_{\mathrm{log}}^{\vee}$ allows us to define a filtration
\[
W_{-2,\cQ^{\vee}_{\mathrm{log}}}\subset W_{-1,\cQ^{\vee}_{\mathrm{log}}}\subset W_{0,\cQ^{\vee}_{\mathrm{log}}}=\cQ^{\vee}_{\mathrm{log}}[n]
\]
by
\[
W_{-2,\cQ^{\vee}_{\mathrm{log}}}\coloneqq T^{\vee}[n]\cong Y^{\vee}\otimes_{\bZ} \mu_{n}, \ \ \ W_{-1,\cQ^{\vee}_{\mathrm{log}}}\coloneqq G^{\vee}[n].
\] 
It follows from the definition of the Weil pairing and the last assertion of Lemma \ref{lem for weil pairing} that 
\[
e_{\cQ_{\mathrm{log}}[n]}(W_{-1,\cQ_{\mathrm{log}}},W_{-2,\cQ^{\vee}_{\mathrm{log}}})=e_{\cQ_{\mathrm{log}}[n]}(W_{-2,\cQ_{\mathrm{log}}},W_{-1,\cQ^{\vee}_{\mathrm{log}}})=0,
\]
and direct computations imply that natural pairings
\begin{align*}
(W_{0,\cQ_{\mathrm{log}}}/W_{-1,\cQ_{\mathrm{log}}})\times W_{-2,\cQ^{\vee}_{\mathrm{log}}}\cong Y/nY\times (Y^{\vee}\otimes_{\bZ} \mu_{n})\to \mu_{n}, \\
(W_{-1,\cQ_{\mathrm{log}}}/W_{-2,\cQ_{\mathrm{log}}})\times (W_{-1,\cQ^{\vee}_{\mathrm{log}}}/W_{-2,\cQ^{\vee}_{\mathrm{log}}})\cong A[n]\times A^{\vee}[n]\to \mu_{n}, \\
W_{-2,\cQ_{\mathrm{log}}}\times (W_{0,\cQ^{\vee}_{\mathrm{log}}}/W_{-1,\cQ^{\vee}_{\mathrm{log}}})\cong (X^{\vee}\otimes_{\bZ} \mu_{n})\times X/nX\to \mu_{n},
\end{align*}
coincide with the induced pairings from $e_{\cQ_{\mathrm{log}}[n]}$ by construction of the Weil pairing. Since the above three pairings are perfect pairings, $e_{\cQ_{\mathrm{log}}[n]}$ is also a perfect pairing. This proves (3).
\end{proof}

Let $(S,\cM_{S})$ be an fs log regular log scheme. Let $U$ be the interior of $(S,\cM_{S})$ and $j\colon U\hookrightarrow S$ be the inclusion map. The pullback along $j$ gives a morphism of sites $(U,\cM_{U})_{\mathrm{kfl}}\to (S,\cM_{S})_{\mathrm{kfl}}$. The associated direct image functor
\[
\mathrm{Shv}((U,\cM_{U})_{\mathrm{kfl}})\to \mathrm{Shv}((S,\cM_{S})_{\mathrm{kfl}})
\]
is denoted by $j_{\mathrm{kfl},*}$. In the same way, the direct image functor
\[
\mathrm{Shv}(U_{\et})\to \mathrm{Shv}(S_{\et})
\]
induced from $j$ is denoted by $j_{\et,*}$.

Let $G$ be a split semi-abelian scheme over $S$ with torus part $T$ and abelian part $A$. 

\begin{lem}\label{description of Glog}
The natural map $G_{\mathrm{log}}\to j_{\mathrm{kfl},*}(G|_{U})$ of sheaves on $(S,\cM_{S})_{\mathrm{kfl}}$ induces an isomorphism of sheaves on $S_{\et}$
    \[
    G_{\mathrm{log}}|_{S_{\et}}\isom j_{\et,*}(G|_{U}).
    \]
\end{lem}

\begin{proof}
First, we treat the case where $G=T$. By working \'{e}tale locally on $S$, we may assume that $T=\bG_{m}$ and that there is an fs chart $P\to \cM_{S}$. Let $S'\in S_{\et}$. We have natural maps
\[
\cM_{S}(S')\to \cO_{S'}(S')\to \cO_{S'}(U\times_{S} S')^{\times}
\]
By Proposition \ref{log reg}(2) and (3), both maps are injective. Hence, the map $\bG_{m,\mathrm{log}}|_{S_{\et}}\to j_{\et,*}(\bG_{m,U})$ in the assertion is injective. Further, Lemma \ref{inv func on interior} implies that the map $\bG_{m,\mathrm{log}}|_{S_{\et}}\to j_{\et,*}(\bG_{m,U})$ is surjective. This proves the assertion then $G=T$.

Next, consider a general split semi-abelian scheme $G$. We have the following commutative diagram of sheaves on $S_{\et}$ whose rows are exact:
    \[
    \begin{tikzcd}
        0 \ar[r] & T_{\mathrm{log}}|_{S_{\et}} \ar[r] \ar[d] & G_{\mathrm{log}}|_{S_{\et}} \ar[r] \ar[d] & A \ar[r] \ar[d] & 0 \\
        0 \ar[r] & j_{\et,*}(T|_{U}) \ar[r] & j_{\et,*}(G|_{U}) \ar[r] & j_{\et,*}(A|_{U}).
    \end{tikzcd}
    \]
    It follows from what we have proved in the previous paragraph that the left vertical map is an isomorphism. By \cite[Ch.I, Proposition 2.7]{fc90}, the right vertical map is also an isomorphism. Therefore, the snake lemma implies that the middle vertical map is also an isomorphism.
\end{proof}

\subsection{Degeneration theory}

In this subsection, we study a relation between semi-abelian schemes, $1$-motives, and log $1$-motives over complete regular local rings.

Let $S=\mathrm{Spec}(R)$ be the spectrum of a complete regular local ring $R$ with unique closed point $s$, and $D$ be a normal crossings divisor on $S$. Let $(S,\cM_{S})$ be the fs log scheme defined by $D$. Set $U\coloneqq S\backslash D$. Note that every torus over $S$ is split after finite \'{e}tale base change.

\begin{dfn}\label{val of section of torsor}
    Let $\cP$ be a $\bG_{m}$-torsor on $S$. Let $\mathrm{Div}(S,U)$ denote the group of Weil divisors of $S$ whose support is contained in $D$. Then taking valuations defined by generic points of irreducible components of $D$ gives a map $\nu\colon \cP(U)\to \mathrm{Div}(X,U)$. 
\end{dfn}

\begin{dfn}\label{categories concerning degeneration}
    We define the following categories. 
    \begin{itemize}
        \item Let $\mathrm{DEG}(S,U)$ be the category of semi-abelian schemes $A$ over $S$ such that $A\times_{S} U$ is an abelian scheme over $U$.
        \item Let $\mathrm{wDD}(S,U)$ be the category of triples $(Y,G,u\colon Y|_{U}\to G|_{U})$ consisting of a locally constant sheaf $Y$ of free $\bZ$-modules of finite rank on $S_{\et}$, a split semi-abelian scheme $G$ over $S$, and a group homomorphism $u\colon Y|_{U}\to G|_{U}$. In the same way as Lemma \ref{symm description of 1mot}, $\mathrm{wDD}(S,U)$ is naturally equivalent to the category of $1$-motives $\cQ_{U}=(X_{U},Y_{U},A_{U},c_{U},c^{\vee}_{U},\tau)$ over $U$ such that the tuple $(X_{U},Y_{U},A_{U},c_{U},c^{\vee}_{U})$ (uniquely) extends to $(X,Y,A,c,c^{\vee})$ over $S$. In particular, $\mathrm{wDD}(S,U)$ is a full subcategory of the category of $1$-motives over $U$.
        \item Let $\mathrm{DD}_{\mathrm{pol}}(S,U)$ be the category of an object $\cQ_{U}=(X_{U},Y_{U},A_{U},c_{U},c^{\vee}_{U},\tau)\in \mathrm{wDD}(S,U)$ equipped with a polarization $\lambda_{U}\colon \cQ_{U}\to \cQ_{U}^{\vee}$ satisfying the following conditions:
        \begin{itemize}
            \item $\lambda_{U}^{\mathrm{ab}}$ extends to a polarization on $A$;
            \item there is a connected finite \'{e}tale cover $S'\to S$ such that the pullback of $Y$ to $S'$ is constant and, for each $y\in Y(S')$, we have
            \[
            \nu(\tau(y,\lambda^{\et}(y)))\in \mathrm{Div}^{+}(S',U')\backslash \{0\},
            \]
            where $U'\coloneqq U\times_{S} S'$ and $\mathrm{Div}^{+}(S',U')$ is the submonoid of $\mathrm{Div}(S',U')$ consisting of effective divisors. Clearly, this condition is independent of the choice of $S'$. 
        \end{itemize}
        Forgetting polarizations gives a functor
        $\mathrm{DD}_{\mathrm{pol}}(S,U)\to \mathrm{wDD}(S,U)$. The essential image of this functor is denoted by $\mathrm{DD}(S,U)$.
        \item Let $\mathrm{wDD}^{\mathrm{log}}(S,U)$ be the category of log $1$-motives over $(S,\cM_{S})$.
        \item Let $\mathrm{DD}^{\mathrm{log}}_{\mathrm{pol}}(S,U)$ be the category of a log $1$-motive $\cQ_{\mathrm{log}}=(X,Y,A,c,c^{\vee},\tau^{\mathrm{log}})$ equipped with a morphism $\lambda\colon \cQ_{\mathrm{log}}\to \cQ_{\mathrm{log}}^{\vee}$ satisfying the following conditions:
        \begin{itemize}
            \item $\lambda^{\mathrm{ab}}$ is a polarization on $A$;
            \item for $y\in Y_{\overline{s}}\backslash \{0\}$, we have $\langle y,\lambda^{\et}(y)\rangle \in (\cM_{S,\overline{s}}/\cO_{S,\overline{s}}^{\times})\backslash \{1\}$, where $\bar{s}$ is a geometric point on $S$ above $s$, and $\langle -,- \rangle $ is the monodromy pairing $Y\times X\to \bG_{m,\mathrm{log}}/\bG_{m}$ (see Definition \ref{monodromy pairing ass with log 1mot}).
        \end{itemize}
        Forgetting $\lambda$ gives a functor
        $\mathrm{DD}^{\mathrm{log}}_{\mathrm{pol}}(S,U)\to \mathrm{wDD}^{\mathrm{log}}(S,U)$. The essential image of this functor is denoted by $\mathrm{DD}^{\mathrm{log}}(S,U)$.
    \end{itemize}
\end{dfn}

\begin{rem}
    The notion of polarizations on log $1$-motives is also defined in {\cite[Definition 2.8]{kkn08b}}. However, we do not use it in this paper. Note that, for an object $(\cQ_{\mathrm{log}},\lambda)$ of the category $\mathrm{DD}^{\mathrm{log}}_{\mathrm{pol}}(S,U)$, the morphism $\lambda$ is not a polarization in the sense of \emph{loc. cit.} unless $S=U$. 
\end{rem}

\begin{thm}[Mumford's degeneration theory, cf.~{\cites{fc90,lan13,mad19}}]\label{deg theory} There is a natural equivalence of categories
    \[
    \mathrm{DEG}(S,U)\simeq \mathrm{DD}(S,U).
    \]
\end{thm}

\begin{proof}
    See \cite[(1.2.2)]{mad19}. 
\end{proof}

By reinterpreting the right hand side in the equivalence in Theorem \ref{deg theory} in terms of log $1$-motives, we obtain the following result.

\begin{thm}\label{log deg theory} 
There are natural equivalences of categories
    \[
    \mathrm{DEG}(S,U)\simeq \mathrm{DD}(S,U)\simeq \mathrm{DD}^{\mathrm{log}}(S,U).
    \]
\end{thm}

\begin{proof}
It is enough to prove the second equivalence. For a locally constant sheaf $Y$ of free $\bZ$-modules of finite rank on $S_{\et}$ and a split semi-abelian scheme $G$ on $S$, giving a group homomorphism $Y|_{U}\to G|_{U}$ is equivalent to giving a homomorphism $Y\to G_{\mathrm{log}}$ by Lemma \ref{description of Glog}. Hence, there is a natural equivalence of categories
    \[
    \mathrm{wDD}(S,U)\simeq \mathrm{wDD}^{\mathrm{log}}(S,U).
    \]
    If we take a connected finite \'{e}tale cover $S'\to S$ such that every irreducible component of $S'-U'$ is regular, we have an isomorphism of monoids $\mathrm{Div}^{+}(S',U')\cong \cM_{S,\overline{s}}/\cO_{S,\overline{s}}^{\times}$, where we put $U'\coloneqq U\times_{S} S'$. Therefore, the above equivalence induces an equivalence
    \[
    \mathrm{DD}_{\mathrm{pol}}(S,U)\simeq \mathrm{DD}^{\mathrm{log}}_{\mathrm{pol}}(S,U),
    \]
    and so we obtain an equivalence
    \[
    \mathrm{DD}(S,U)\simeq \mathrm{DD}^{\mathrm{log}}(S,U).
    \]
\end{proof}

\begin{prop}\label{compatibility for torsion point}
Let $A\in \mathrm{DEG}(S,U)$. Let $\cQ_{U}$ (resp. $\cQ_{\mathrm{log}}$) be the object of $\mathrm{DD}(S,U)$ (resp. $\mathrm{DD}^{\mathrm{log}}(S,U)$) corresponding to $A$ via the equivalences in Theorem \ref{log deg theory}. Then there are natural isomorphisms of finite flat group schemes over $U$
    \[
    A|_{U}[n]\cong \cQ_{U}[n]\cong \cQ_{\mathrm{log}}[n]|_{U}
    \]
    for any integer $n\geq 1$.
\end{prop}

\begin{proof}
    We have a natural isomorphism $A|_{U}[n]\cong \cQ_{U}[n]$ (see \cite[Ch.III, Corollary 7.3]{fc90} or \cite[(1.2.2.1)]{mad19}). Since the restriction of $\cQ_{\mathrm{log}}$ to $U$ coincides with $\cQ_{U}$ by construction, we have natural isomorphisms
    \[
    \cQ_{U}[n]\cong (\cQ_{\mathrm{log}})|_{U}[n]\cong \cQ_{\mathrm{log}}[n]|_{U}.
    \]
\end{proof}

\section{The proof of the main theorems}\label{sec4}

\begin{lem}\label{fully faithfulness dvr case}
Let $R$ be a discrete valuation ring with fraction field $K$, and $\cM_{R}$ be the log structure on $\mathrm{Spec}(R)$ defined by the unique closed point. 
\begin{enumerate}
    \item Let $n\geq 1$ be an integer invertible in $R$. Let $G_{1}$ and $G_{2}$ be weak log finite group schemes over $(\mathrm{Spec}(R),\cM_{R})$ killed by $n$. Then the restriction map
    \[
    \mathrm{Hom}(G_{1},G_{2})\to \mathrm{Hom}(G_{1}|_{K},G_{2}|_{K})
    \]
    is an isomorphism.
    \item Let $p$ be a prime number. Let $G_{1}$ and $G_{2}$ be log $p$-divisible groups over $(\mathrm{Spec}(R),\cM_{R})$. Then the restriction map
    \[
    \mathrm{Hom}(G_{1},G_{2})\to \mathrm{Hom}(G_{1}|_{K},G_{2}|_{K})
    \]
    is an isomorphism.
\end{enumerate}
\end{lem}

\begin{proof}
    (1) follows from Lemma \ref{log fin grp sch and lcf sheaf} and the surjectivity of 
\[
\mathrm{Gal}(\overline{K}/K)\to \pi_{1,\mathrm{k\et}}(\mathrm{Spec}(R),\cM_{R}).
\]
(2) is nothing but the log version of the theorem of Tate and de Jong (\cite[Theorem 5.19]{bwz24}). Note that, although \cite[Theorem 5.19]{bwz24} assumes that $K$ is of mixed characteristic $(0,p)$, and that the residue field of $R$ is perfect, the fully faithfulness part is essentially proved in \cite[Lemma 4.8]{bwz24}, in which the argument works without such an additional assumption.
\end{proof}

\begin{prop}\label{mainthm dvr case}
    Let notations be as in Lemma \ref{fully faithfulness dvr case}. Let $A$ be a semi-abelian scheme over $R$ with $A_{K}\coloneqq A\otimes_{R} K$ being an abelian variety over $K$.  
    \begin{enumerate}
        \item Let $n\geq 1$ be an integer invertible in $R$. Then the finite group scheme $A_{K}[n]$ over $K$ uniquely extends to a log finite group scheme over $(\mathrm{Spec}(R),\cM_{R})$.
        \item Let $p$ be a prime number. Then the $p$-divisible group $A_{K}[p^{\infty}]$ over $K$ uniquely extends to a log $p$-divisible group over $(\mathrm{Spec}(R),\cM_{R})$.
    \end{enumerate}
\end{prop}

\begin{proof}
For both assertions, the uniqueness follows from Lemma \ref{fully faithfulness dvr case}. By Proposition \ref{bl gluing for log fin grp}, we may assume that $R$ is complete. Let $\cQ_{\mathrm{log}}$ be the log $1$-motive on $(\mathrm{Spec}(R),\cM_{R})$ corresponding to $A\in \mathrm{DEG}(R,K)$ via the equivalence in Theorem \ref{log deg theory}. Then the log finite group scheme $\cQ_{\mathrm{log}}[n]$ and the log $p$-divisible group $\cQ_{\mathrm{log}}[p^{\infty}]$ are the desired extensions.
\end{proof}

\begin{lem}\label{lem on purity for morphisms of kfl vect bdles}
    Let $f\colon X\to Y$ be a flat morphism from a (not necessarily locally noetherian) scheme $X$ to a locally noetherian normal scheme $Y$. Let $U$ be a dense open subset of $Y$ containing all points of codimension $1$. Then the restriction functor
    \[
    \mathrm{Vect}(X)\to \mathrm{Vect}(f^{-1}(U))
    \]
    is fully faithful.
\end{lem}

\begin{proof}
    By taking internal homomorphisms, the problem is reduced to showing that, for a vector bundle $\cE$ on $X$, the restriction map
    \[
    \Gamma(X,\cE)\to \Gamma(f^{-1}(U),\cE)
    \]
    is an isomorphism. Let $i\colon f^{-1}(U)\hookrightarrow X$ and $j\colon U\hookrightarrow Y$ be natural open immersions. Then we have isomorphisms of $\cO_{X}$-modules
    \begin{align*}
        i_{*}i^{*}\cE\cong \cE\otimes i_{*}\cO_{f^{-1}(U)}\cong \cE\otimes f^{*}j_{*}\cO_{U}\cong \cE\otimes f^{*}\cO_{Y}\cong \cE,
    \end{align*}
    where the first isomorphism is the projection formula, the second one is the flat base change, and the third one follows from the assumption that $U$ is an open subset of a locally noetherian normal scheme $Y$ containing all points of codimension $1$. Taking global sections on both sides, we obtain the statement.
\end{proof}

\begin{prop}\label{purity for hom of kfl vect bdle}
    Let $f\colon (X,\cM_{X})\to (Y,\cM_{Y})$ be a strict flat morphism from a (not necessarily locally noetherian) fs log scheme $(X,\cM_{X})$ to a log regular log scheme $(Y,\cM_{Y})$. Let $U$ be a dense open subset of $Y$ containing all points of codimension $1$. Then the restriction functor
    \[
    \mathrm{Vect}_{\mathrm{kfl}}(X,\cM_{X})\to \mathrm{Vect}_{\mathrm{kfl}}(f^{-1}(U),\cM_{f^{-1}(U)})
    \]
    is fully faithful, where $\cM_{f^{-1}(U)}$ is the pullback log structure of $\cM_{X}$.
\end{prop}

\begin{proof}
    Let $\cE_{1},\cE_{2}$ be kfl vector bundles on $(X,\cM_{X})$. We shall prove that the restriction map
    \[
    \mathrm{Hom}(\cE_{1},\cE_{2})\to \mathrm{Hom}(\cE_{1}|_{f^{-1}(U)},\cE_{2}|_{f^{-1}(U)})
    \]
    is an isomorphism. By the limit argument (cf.~\cite[Appendix]{ino23}), we may assume that $Y$ is a spectrum of a strict local ring. Let $y\in Y$ be a unique closed point. Take a chart $P\to \cM_{Y}$ such that $P\to \cM_{Y,y}/\cO_{Y,y}^{\times}$ is an isomorphism. By Lemma \ref{kfl vect bdle is classical after n-power ext}, we can take an integer $n\geq 1$ such that the pullback of $\cE_{i}$ to $(X',\cM_{X'})\coloneqq (X,\cM_{X})\otimes_{(\bZ[P],P)^{a}} (\bZ[P^{1/n}],P^{1/n})^{a}$ is classical for $i=1,2$. Let $(Y',\cM_{Y'})\coloneqq (Y,\cM_{Y})\otimes_{(\bZ[P],P)^{a}} (\bZ[P^{1/n}],P^{1/n})^{a}$. By Lemma \ref{log reg stab}, $(Y',\cM_{Y'})$ is log regular, and so $Y'$ is normal by Proposition \ref{log reg}(2). Let $(X'',\cM_{X''})$ denote the self-saturated fiber product of $(X',\cM_{X'})$ over $(X,\cM_{X})$. Let $V'$ (resp. $V''$) (resp. $U'$) be the preimage of $U$ in $X'$ (resp. $X''$) (resp. $Y'$). Since $Y'\to Y$ corresponds to an integral extension of normal domains by Lemma \ref{log reg}(2) and the separatedness of $\cO_{(Y,\cM_{Y})}$, $U'$ is also a dense open subset of $Y'$ containing all points of codimension $1$. Applying Lemma \ref{lem on purity for morphisms of kfl vect bdles} to flat and qcqs morphisms of schemes $X'\to Y'$ and $X''\to Y'$ and the open subset $U'\subset Y'$, we conclude that the restriction maps
    \begin{align*}
        \mathrm{Hom}(\cE_{1}|_{(X',\cM_{X'})},\cE_{2}|_{(X',\cM_{X'})})&\to \mathrm{Hom}(\cE_{1}|_{(V',\cM_{V'})},\cE_{2}|_{(V',\cM_{V'})}), \\
        \mathrm{Hom}(\cE_{1}|_{(X'',\cM_{X''})},\cE_{2}|_{(X'',\cM_{X''})})&\to \mathrm{Hom}(\cE_{1}|_{(V'',\cM_{V''})},\cE_{2}|_{(V'',\cM_{V''})})
    \end{align*}
    are isomorphisms. Therefore, the claim follows from kfl descent. 
\end{proof}

\begin{cor}[Purity for homomorphisms of weak log fnite group schemes]\label{purity for hom of weak log fin grp}
Under the assumption of Proposition \ref{purity for hom of kfl vect bdle}, the restriction functor
\[
\mathrm{wFin}(X,\cM_{X})\to \mathrm{wFin}(f^{-1}(U),\cM_{f^{-1}(U)})
\]
is fully faithful.
\end{cor}

\begin{proof}
    This follows from Lemma \ref{str sheaf of log fin grp sch} and Proposition \ref{purity for hom of kfl vect bdle}.
\end{proof}

\begin{thm}\label{mainthm fin ver}
    Let $n\geq 1$ be an integer. Let $(X,\cM_{X})$ be an fs log scheme defined by a locally noetherian regular scheme $X$ with a normal crossings divisor $D$, and $A$ be a semi-abelian scheme over $X$. Let $U\coloneqq X\backslash D$. Suppose that $A_{U}\coloneqq A\times_{X} U$ is an abelian scheme over $U$ and that $D\otimes_{\bZ} \bZ[1/n]$ is dense in $D$. Then the finite flat group scheme $A_{U}[n]$ over $U$ uniquely extends to a log finite group scheme $A[n]^{\mathrm{log}}$ over $(X,\cM_{X})$.
\end{thm}

\begin{proof}
First, we prove the following claim: for weak log finite group schemes $G_{1}$ and $G_{2}$ over $(X,\cM_{X})$ killed by $n$, the restriction map
\[
\mathrm{Hom}(G_{1},G_{2})\to \mathrm{Hom}(G_{1}|_{U},G_{2}|_{U})
\]
is an isomorphism. Take a homomorphism $f_{U}\colon G_{1}|_{U}\to G_{2}|_{U}$. By the assumption, $n$ is invertible in $\cO_{X,\eta}$ for each generic point $\eta$ of $D$. Hence, there exist an open subset $V$ of $X$ containing $U$ with $\mathrm{codim}_{X}(X\backslash V)\geq 2$ and an extension $f_{V}\colon G_{1}|_{V}\to G_{2}|_{V}$ of $f_{U}$ by Lemma \ref{fully faithfulness dvr case}(1) and the limit argument. Then $f_{V}$ uniquely extends to a homomorphism $f\colon G_{1}\to G_{2}$ by Corollary \ref{purity for hom of weak log fin grp}. This argument also shows that $f$ is a unique extension of $f_{U}$.

We turn to proving Theorem \ref{mainthm fin ver}. Since the uniqueness of the extension follows from the claim in the previous paragraph, we may work Zariski locally on $X$. Hence, the limit argument allows us to assume that $X=\mathrm{Spec}(R)$ for a local ring $R$. By the assumption, $n$ is invertible in $\cO_{X,\eta}$ for each generic point $\eta$ of $D$. Hence, there exist an open subset $V$ of $X$ containing $U$ with $\mathrm{codim}_{X}(X\backslash V)\geq 2$ and a log finite group scheme $A_{V}[n]^{\mathrm{log}}$ over $(V,\cM_{V})$ extending $A_{U}[n]$ by Proposition \ref{mainthm dvr case}(1) and the limit argument. It is enough to extend $A_{V}[n]^{\mathrm{log}}$ to a log finite group scheme over $(X,\cM_{X})$. Let $X'\coloneqq \mathrm{Spec}(\widehat{R})$ and $X''\coloneqq X'\times_{X} X'$. Let $p_{i}\colon (X'',\cM_{X''})\to (X',\cM_{X'})$ be the projection maps for $i=1,2$. Let $U'$ (resp.~$V'$) denotes the preimage of $U$ (resp.~$V$) via $X'\to X$. Let $\cQ_{\mathrm{log}}$ be the object of $\mathrm{DD}^{\mathrm{log}}(X',U')$ corresponding to $A\times_{X} X'\in \mathrm{DEG}(X',U')$. By Proposition \ref{compatibility for torsion point}, the log finite group scheme $\cQ_{\mathrm{log}}[n]$ is an extension of $(A\times_{X} U')[n]$. By the claim in the previous paragraph, we have an isomorphism $\cQ_{\mathrm{log}}[n]|_{V'}\cong A_{V}[n]^{\mathrm{log}}|_{V'}$. Applying Corollary \ref{purity for hom of weak log fin grp} to the strict flat map $(X'',\cM_{X''})\to (X,\cM_{X})$ and the open set $V\subset X$, we obtain a unique isomorphism $p_{1}^{*}\cQ_{\mathrm{log}}[n]\cong p_{2}^{*}\cQ_{\mathrm{log}}[n]$ extending the isomorphism
\[
(p_{1}^{*}\cQ_{\mathrm{log}}[n])|_{V''}\cong A_{V}[n]^{\mathrm{log}}|_{V''}\cong 
p_{2}^{*}\cQ_{\mathrm{log}}[n]|_{V''},
\]
where $V''$ is the preimage of $V$ via $X''\to X$. This defines descent datum of $\cQ_{\mathrm{log}}[n]$, and we obtain a log finite group scheme $A[n]^{\mathrm{log}}$ extending $A_{V}[n]^{\mathrm{log}}$ by strict fpqc descent for finite Kummer log flat schemes (\cite[Theorem 7.1 and Theorem 8.1]{kat21}). This finishes the proof.
\end{proof}

\begin{thm}\label{mainthm BT ver}
    Let $(X,\cM_{X})$ be an fs log scheme defined by a locally noetherian regular scheme $X$ with a normal crossings divisor $D$, and $A$ be a semi-abelian scheme over $X$. Let $U\coloneqq X\backslash D$. Suppose that $A_{U}$ is an abelian scheme over $U$. Then the $p$-divisible group $A_{U}[p^{\infty}]$ over $U$ uniquely extends to a log $p$-divisible group $A^{\mathrm{log}}[p^{\infty}]$ over $(X,\cM_{X})$. 
\end{thm}

\begin{proof}
    Although almost all arguments of Theorem \ref{mainthm fin ver} work in this setting, we have to notice that we need to pass to finite levels when we use the limit argument. 

For each generic point $\eta$ of $D$, there exists a unique log $p$-divisible group $A_{\cO_{X,\eta}}[p^{\infty}]^{\mathrm{log}}$ over $(\mathrm{Spec}(\cO_{X,\eta}),\cM_{\cO_{X,\eta}})$ extending $A_{U}[p^{\infty}]|_{\mathrm{Spec}(K(\eta))}$ by Lemma \ref{mainthm dvr case}, where $K(\eta)$ denotes the fraction field of $\cO_{X,\eta}$. For every $n\geq 1$, by the same argument as Theorem \ref{mainthm fin ver}, there exists a unique log finite group $A[p^{n}]^{\mathrm{log}}$ over $(X,\cM_{X})$ with compatible isomorphisms $A[p^{n}]^{\mathrm{log}}|_{U}\cong A_{U}[p^{n}]$ and $A[p^{n}]^{\mathrm{log}}|_{(\mathrm{Spec}(\cO_{X,\eta}),\cM_{\cO_{X,\eta}})}\cong A_{\cO_{X,\eta}}[p^{n}]$ for each generic point $\eta$ of $D$.

Natural inclusion maps $A_{U}[p^{n}]\hookrightarrow A_{U}[p^{n+1}]$ and $A_{\cO_{X,\eta}}[p^{n}]\hookrightarrow A_{\cO_{X,\eta}}[p^{n+1}]$ uniquely extend to a homomorphism $A[p^{n}]^{\mathrm{log}}\to A[p^{n+1}]^{\mathrm{log}}$ for every $n\geq 1$ by the limit argument and Corollary \ref{purity for hom of weak log fin grp}. It is enough to check that $A[p^{\infty}]^{\mathrm{log}}\coloneqq \displaystyle \varinjlim_{n\geq 1}
    A[p^{n}]^{\mathrm{log}}$ is a log $p$-divisible group with
    \[
    A[p^{n}]^{\mathrm{log}}\cong \mathrm{Ker}(\times p^{n}\colon A[p^{\infty}]^{\mathrm{log}}\to A[p^{\infty}]^{\mathrm{log}}).
    \]
To check this, it suffices show that it is so after taking the base change to $(\mathrm{Spec}(\widehat{\cO}_{X,x}),\cM_{\widehat{\cO}_{X,x}})$ for every $x\in X$. Hence, we may assume that $X=\mathrm{Spec}(R)$ for a complete regular local ring $R$. Let $\cQ_{\mathrm{log}}$ be the object of $\mathrm{DD}^{\mathrm{log}}(X,U)$ corresponding to $A$ under the equivalence in Theorem \ref{log deg theory}. By the uniqueness of extension, $A[p^{n}]^{\mathrm{log}}$ is isomorphic to $\cQ_{\mathrm{log}}[p^{n}]$. Therefore, the claim follows from Proposition \ref{log fin ass to log 1mot}(3).
\end{proof}


\begin{thebibliography}{99}
\bibitem[BL95]{bl95} A.\ Beauville, Y.\ Laszlo: \textit{Un lemme de descente}, 
C.\  R.\  Acad.\ Sci.\ Paris S\'{e}r.\ I Math.\ 320 (1995), no.\ 3, 335-340.
\bibitem[BWZ24]{bwz24} A.\ Bertapelle, S.\ Wang, H.\ Zhao: \textit{Log $p$-divisible groups and semi-stable representations}, 2024,  arXiv:2302.11030.
\bibitem[Del74]{del74} P.\ Deligne: \textit{Th\'{e}orie de Hodge, III}, Inst.\ Hautes.\ \'{E}tudes Sci.\ Publ.\ Math.\ No.\ 44 (1974), 5-77.
\bibitem[FC90]{fc90} G.\ Faltings, C.-L.\ Chai: \textit{Degeneration of abelian varieties},
Ergeb.\ Math.\ Grenzgeb.\ (3), 22, Springer-Verlag, Berlin, 1990.
\bibitem[Ino23]{ino23} K.\ Inoue: \textit{Slope filtrations of log $p$-divisible groups}, 2023, arXiv:2307.04629.
\bibitem[Ino25]{ino25} K.\ Inoue \textit{Log prismatic Dieudonn\'{e} theory and its application to Shimura varieties}, 2025, arXiv:2503.07379.
\bibitem[KKN08a]{kkn08a} T.\ Kajiwara, K.\ Kato, C.\ Nakayama: \textit{Logarithmic abelian varieties. I. Complex analytic theory}, J.\ Math.\ Sci.\ Univ.\ Tokyo 15 (2008), no.\ 1, 69-193.
\bibitem[KKN08b]{kkn08b} T.\ Kajiwara, K.\ Kato, C.\ Nakayama: \textit{Logarithmic abelian varieties}, Nagoya Math.\ J.\ 189 (2008), 63-138.
\bibitem[KKN15]{kkn15} T.\ Kajiwara, K.\ Kato, C.\ Nakayama: \textit{Logarithmic abelian varieties, Part IV: Proper models}, Nagoya Math.\ J.\ 219 (2015), 9-63.
\bibitem[KKN18]{kkn18} T.\ Kajiwara, K.\ Kato, C.\ Nakayama: \textit{Logarithmic abelian varieties, Part V: Projective models}, Yokohama Math.\ J.\ 64 (2018), 21-82.
\bibitem[KKN19]{kkn19} T.\ Kajiwara, K.\ Kato, C.\ Nakayama: \textit{Logarithmic abelian varieties, Part VI: Local moduli and GAGF}, Yokohama Math.\ J.\ 65 (2019), 53-75.
\bibitem[KKN21]{kkn21} T.\ Kajiwara, K.\ Kato, C.\ Nakayama: \textit{Logarithmic abelian varieties, part VII: moduli}, Yokohama Math.\ J.\ 67 (2021), 9-48.
\bibitem[KKN22]{kkn22} T.\ Kajiwara, K.\ Kato, C.\ Nakayama: \textit{Moduli of logarithmic abelian varieties with PEL structure}, 2022, arXiv:2205.10985.
\bibitem[Kat89]{kat89} K.\ Kato: \textit{Logarithmic structures of Fontaine-Illusie}, Algebraic analysis, geometry, and number theory (Baltimore, MD, 1988), 191-224, Johns Hopkins Univ.\ Press, Baltimore, MD, 1989.
\bibitem[Kat94]{kat94} K.\ Kato: \textit{Toric singularities}, Amer.\ J.\ Math.\ 116 (1994), no.\ 5, 1073-1099.
\bibitem[Kat21]{kat21} K.\ Kato: \textit{Logarithmic structures of Fontaine-Illusie. II-Logarithmic flat topology}. Tokyo J.\ Math.\ 44 (2021), no.\ 1, 125-155.
\bibitem[Kat23]{kat23} K.\ Kato: \textit{Logarithmic Dieudonn\'{e} theory}, 2023, arXiv:2306.13943.
\bibitem[Lan13]{lan13} K.-W.\ Lan: \textit{Arithmetic compactifications of PEL-type Shimura varieties}, London Mathematical Society monographs 36, Princeton Univ.\ Press, 2013.
\bibitem[Mad19]{mad19} K.\ Madapusi: \textit{Toroidal compactifications of integral models of Shimura varieties of Hodge type}, Ann.\ Sci.\ \'{E}c.\ Norm.\ Sup\'{e}r.\ 52 (2019), no.\ 2, 393-514.
\bibitem[Ogu18]{ogu18} A.\ Ogus: \textit{Lectures on logarithmic algebraic geometry}, Cambridge Studies in Advanced Mathematics, 178.\ Cambridge University Press, Cambridge, 2018.
\bibitem[WZ24]{wz24} M.\ W\"{u}rthen, H.\ Zhao: \textit{Log $p$-divisible groups associated with log $1$-motives}, Canad.\ J.\ Math.\ 76 (2024), no.\ 3, 946-983.
\bibitem[Zha21]{zha21} H.\ Zhao: \textit{Extending tamely ramified strict $1$-motives into k\'{e}t log $1$-motives}, Forum Math.\ Sigma 9 (2021).
\bibitem[SGA7-I]{sga7} \textit{Groupes de monodromie en g\'{e}om\'{e}trie alg\'{e}brique. I}, Lecture Notes in Math., 288, Springer, Berlin, 1972.

\end{thebibliography}
\end{document}